\documentclass[12pt]{article}
\title{On definable J-sets}
\author{Zhentao Zhang}

\date{}

\usepackage{amsmath, amssymb, amsthm}    	
\usepackage{fullpage} 	
\usepackage{amscd}
\usepackage{hyperref}
\usepackage[all]{xy}
\usepackage{centernot}
\usepackage{color,xcolor} 
\usepackage{ulem}

\DeclareMathOperator*{\forkindep}{\raise0.2ex\hbox{\ooalign{\hidewidth$\vert$\hidewidth\cr\raise-0.9ex\hbox{$\smile$}}}}

\newcommand{\Gen}{\operatorname{Gen}}
\newcommand{\WGen}{\operatorname{WGen}}
\newcommand{\AP}{\operatorname{AP}}

\newcommand{\A}{\mathbb{A}}

\newcommand{\N}{\mathbb{N}}

\newcommand{\Z}{\mathbb{Z}}
\newcommand{\Q}{\mathbb{Q}}

\newcommand{\pCF}{p\mathrm{CF}}

\newcommand{\supp}{\mathrm{supp}}

\newcommand{\jj}{\mathfrak{j}}
\newcommand{\Pf}{\mathcal{P}_\mathrm{fin}}
\newcommand{\JJ}{\mathrm{J}}
\newcommand{\ext}{\mathrm{ext}}
\newcommand{\V}{\mathbb{V}}

\newtheorem{theorem}{Theorem}[section] 
\newtheorem{lemma}[theorem]{Lemma}

\newtheorem{coro}[theorem]{Corollary}
\newtheorem{fact}[theorem]{Fact}

\newtheorem{question}[theorem]{Question}
\newtheorem{prop}[theorem]{Proposition}
\newtheorem{proposition-eh}[theorem]{Proposition(?)}
\newtheorem*{theorem-star}{Main Theorem}
\newtheorem*{theorem-star-A}{Theorem A}
\newtheorem*{theorem-star-B}{Theorem B}
\newtheorem*{conjecture-star}{Conjecture}
\newtheorem*{lemma-star}{Lemma}
\newtheorem*{claim-star}{Claim}

\theoremstyle{definition}
\newtheorem{definition}[theorem]{Definition}

\newtheorem{remark}[theorem]{Remark}

\newcommand{\tp}{\mathrm{tp}}
\newcommand{\M}{\mathbb{M}}
\newcommand{\cO}{\mathcal{O}}
\newcommand{\Th}{\mathrm{Th}}

\newcommand{\Ga}{\mathbb{G}_\mathrm{a}}

\newenvironment{claimproof}[1][\proofname]
               {
                 \proof[#1]
                 
               }
               {
                 \endproof
               }

\begin{document}
\maketitle
\begin{abstract}
We study definable J-sets for definable groups and compare them with weakly generic sets. We show that the property that J-sets coincide with weakly generic sets is invariant on enough saturated models, and hence a model-theoretical property. We have positive results
for superstable commutative groups and some easy examples in $\pCF$. We also give an example for noncoincidence. 
\end{abstract}

\section{Introduction}

In \cite{Hindman}, J-sets are studied in combinatorics for semigroups and are regarded as ``large'' sets. We can study them among definable subsets of a definable group $G$. The method for Stone-Cech compactification can be replaced by the type space of externally definable subsets of $G$ in \cite{Newelski}. It is natural to compare J-sets with weakly generic sets given in \cite{Newelski} which are also about largeness. 

In Section \ref{section-basic}, we give some basic properties on J-sets. Let $X$ be a definable subset of $G$. We can study whether $X$ is a J-set of $G$ on different models. Like Lemma 2.1 \cite{Newelski} for weakly generic sets, the question has the same answer on $\aleph_0$-saturated models. Hence, whether J-sets coincide with weakly generic sets is invariant for $\aleph_0$-saturated models. Thus, the property that J-sets coincide with weakly generic sets on enough saturated models (for convenience, called the \textbf{J-property}) is for the theory of $G$. 

In Section \ref{section-positive}, we show that J-property is related to the classification theory: superstable commutative groups have J-property. 
We also show the additive groups of the field and the valued ring in $\pCF$, the theory of $\Q_p$ in the language of rings, have J-property.  

In Section \ref{section-example}, following \cite{example-neq}, we give an example of a commutative group without J-property which is an expansion of $(\Z,+)$ by a unary predicate. But the theory is not stable.

For notation, we always let $M$ be a structure and $G$ a group definable over $M$. We usually work in a monster model $\M$ of $T=\Th(M)$, the theory of $M$. 
We say that a set $X$ is defined in $M$ if $X=X(M)$ is definable in the structure $M$. If we do not say in which a definable set $X$ is defined, it is considered in $\M$ and $X=X(\M)$. We always write $G=G(\M)$.  We say that $G$ is stable/superstable if its induced theory from $T$ is, and for convenience, we assume $T$ is stable/superstable in this case.
Sometimes, we assume that $0\in\N$, while sometimes, to start from $1$, we assume that $0\notin \N$.

\section*{Acknowledgements}
Thanks to Teng Zhang for beneficial discussions and informing me of J-sets. 

\section{Basic properties}\label{section-basic}

We recall the definition of J-sets from Definition 14.14.1 \cite{Hindman}:

\begin{definition}
A subset $X$ of $G(M)$ is called a \textbf{J-set of $G$ in $M$}, if for every $F\in \Pf(G(M)^\N):=\{F\subset G(M)^\N: |F|<\aleph_0\}$, there are $m\in\N$, $t\in \jj_m:=\{s\in \N^m: s(i)<s(j+1) \text{ for all }j\}$ and $a\in G(M)^{m+1}$ such that for every $f\in F$,
$$\chi(m,a,t,f):=a(1)f(t(1))a(2)f(t(2))\cdots a(m-1)f(t(m-1))a(m+1)\in X$$ 

We call a type on $G$ over $M$ is a \textbf{J-type}, if every $X\in p$ is a J-set in $M$.
\end{definition}

\begin{remark}
It is easy to see that if $X$ is a J-set of $G$ in $M$ and $g\in G(M)$, then $g  X=\{gx:x\in X\}$ and $X g:=\{xg:x\in X\}$ are also J-sets of $G$ in $M$.
\end{remark}

Note that (see \cite{Hindman} Lemma 14.14.6)
if $X\cup Y$ is a J-set of $G$ in $M$, then at least one of $X$ and $Y$ is a J-set of $G$ in $M$. Hence, every F.I.P (finite intersection property) family of definable J-set of $G$ in $M$ can be extended to a J-type over $M$ and J-types always exist.

Let $S^\ext_G(M)$ be the type space of externally definable subsets of $G(M)$ given in \cite{Newelski} Section 4. It is known that $(S^\ext_G(M),*)$ is a compact right topological semigroup where for $p,q\in S^\ext_G(M)$ and externally definable $X\subset G(M)$,
$$X\in p*q  \text{ } \text{ } \text{ iff }  \text{ }  \text{ } \{g\in G(M): g^{-1}X \in q\}\in p$$
Moreover, $G(M)$ naturally (left) acts on $S^\ext_G(M)$ and $(S^\ext_G(M),*)$ is the envelope semigroup of the action. Then every $G(M)$-subflow (i.e. a closed space which is closed under the $G(M)$-action) is a left ideal. Let $\AP_G(M)$ be the minimal two-side ideal of $S^\ext_G(M)$ whose elements are called \textbf{almost periodic} in \cite{Newelski}. It is the union of all minimal $G(M)$-subflows. Let $\WGen_G(M)$ be the closure of $\AP_G(M)$ whose elements are called \textbf{weakly generic} in \cite{Newelski}. Here, the definition of weak genericity is from Corollary 1.8 \cite{Newelski}. 

Originally, from Definition 1.1 \cite{Newelski}, a subset of $X$ of $G(M)$ is \textbf{generic} if $G(M)$ can be covered by finitely many (left) translates of $X$, and a subset of $X$ of $G(M)$ is \textbf{weakly generic} if $X\cup Y$ is generic for some non-generic (definable, see Remark 1.2 \cite{Newelski}) $Y$. A type on $G$ over $M$ is (weakly) generic if every its element is (weakly) generic.  Let $\Gen_G(M)$ be the subspace of generic types in $S^\ext_G(M)$. 
Note that generic types always exist and play an important role in stable groups. But in general, generic types may not exist, while weakly generic types always exist.
Moreover, when $\Gen_G(M)\neq\emptyset$, by Corollary 1.9 \cite{Newelski}, $\WGen_G(M)$ is the only minimal $G(M)$-subflow and then $\AP_G(M)=\WGen_G(M)=\Gen_G(M)$.

We denote the subspace of J-types in $S^\ext_G(M)$ by $\JJ_G(M)$. Then like Theorem 14.14.4 \cite{Hindman}, we have

\begin{prop}
$\JJ_G(M)$ is a compact two-sided ideal of $S^\ext_G(M)$.
\end{prop}
\begin{proof}
Clearly,  $\JJ_G(M)$ is the intersection of $[X]:=\{p\in S^\ext_G(M): X\in p\}$ for all externally definable $X$ with $G(M)\backslash X$ not a J-set for $G$ in $M$. Hence, $\JJ_G(M)$ is closed in $S^\ext_G(M)$ and is compact. As $\JJ_G(M)$ is closed by the $G(M)$-action, it is a left ideal of $S^\ext_G(M)$.

Let $X\in p*q$. We show that $X$ is a J-set in $M$. Let $F\in \Pf(G(M)^\N)$.
As $\{g\in G(M): g^{-1}X\in q\}$ is in $p$ and hence a J-set of $G$ in $M$, there are $m\in\N$, $t\in \jj_m$ and $a\in G(M)^{m+1}$ such that for every $f\in F$,
$\chi(m,a,t,f)^{-1}\cdot X\in q$. Then $\bigcap_{f\in F}\chi(m,a,t,f)^{-1}\cdot X\in q$ is nonempty and we take $b$ from it. Then $\chi(m,a,t,f)b\in X$ for every $f\in F$. We let $c\in G(M)^{m=1}$ with $c(i)=a(i)$ for $i\leq m$ and $c(m+1)=a(m+1)b$. Then $\chi(m,c,t,f)\in X$ for every $f\in F$.
\end{proof}

As $\JJ_G(M)$ is a two-sided ideal of $S^\ext_G(M)$, $\JJ_G(G)\supset \AP_G(M)$. Then by closeness of $\JJ_G(M)$ in $S^\ext_G(M)$, we have $\JJ_G(M)\supset \WGen_G(M)$. In summary, we have that

\begin{coro}
Every definable weakly generic sets of $G$ in $M$ is a J-set of $G$ in $M$.
\end{coro}

Now we study the J-sets in different models:

\begin{prop}
Assume that $M\prec N$ and $X\subset G$ is definable in $M$.
If $X(N)$ is a J-set of $G$ in $N$, then $X$ is a J-set of $G$ in $M$. Moreover, if $M$ is $\aleph_0$-saturated, the inverse holds. 
\end{prop}
\begin{proof}
Let $F\in \Pf(G(M)^\N)$ be arbitrary. Then there are $m\in \N$, $a\in G(N)^{m+1}\}$ and $t\in\jj_m$ such that $\chi(m,a,t,f)\in X$ for every $f\in F$. Then
$$N\models\exists y_1\dots y_{m+1}\bigwedge_{f\in F} \text{``}\chi(m,y_1,\dots, y_{m+1}, t,f)\in X\text{''}$$
As $M\prec N$ and the statement is over $M$, we have
$$M\models\exists y_1\dots y_{m+1}\bigwedge_{f\in F} \text{``}\chi(m,y_1,\dots, y_{m+1}, t,f)\in X\text{''}$$
Then there is $a'\in G(M)^{m+1}$ such that $\chi(m,a',t,f)\in X$ for every $f\in F$. 

Now we show the ``moreover'' part. If $X(N)$ is not a J-set of $G$ in $N$, then there are $f_1,\dots, f_n\in G(N)^\N$ such that for every $m\in\N$ and $t\in \jj_m$,
$$N\models \forall y_1\dots y_{m+1} \bigvee_{i=1}^n \text{``}\chi(m,y_1,\dots, y_{m+1},t,f_i)\notin X\text{''}$$ 
Let $C$ be a finite set of $M$ over which $X$ and $G$ are defined. As $M$ is $\aleph_0$-saturated, after reordering the index set of $f_{i,j}$ by $(\N,<)$, we can find $g_{i,j}\in G(M)$ inductively such that $$\tp(g_{i,j}:1\leq i\leq n, j\in \N/C)=\tp(f_{i,j}:1\leq i\leq n, j\in \N/C)$$
 Let $g_i\in G(M)^\N$ with $g_i(j)=g_{i,j}$. Then for every $m\in\N$ and $t\in \jj_m$,
$$M\models \forall y_1\dots y_{m+1} \bigvee_{i=1}^n \text{``}\chi(m,y_1,\dots, y_{m+1},t,g_i)\notin X\text{''}$$  which is a contradiction.
\end{proof}

And, for weakly generic sets, we have
\begin{fact}[Lemma 2.1 \cite{Newelski}]
Assume that $M\prec N$ and $X\subset G$ is definable in $M$.
If $X(M)$ is weakly generic for $G(M)$, then $X(N)$ is weakly generic for $G(N)$. Moreover, if $M$ is $\aleph_0$-saturated, the inverse holds. 
\end{fact}

Then we have

\begin{prop}
Assume that $M\prec N$ and they are both $\aleph_0$-saturated. Then J-sets of $G(M)$ coincide with weakly generic sets of $G(M)$ iff J-sets of $G(N)$ coincide with weakly generic sets of $G(N)$.
\end{prop}
\begin{proof}
The direction $\Leftarrow$ is obvious from the above results. 

For the other direction, we assume that J-sets of $G(M)$ coincide with weakly generic sets of $G(M)$. Since the direction $\Leftarrow$ is known, we may assume that $N=\M$.
For convenience, we assume that $G$ is $0$-definable.
If the result fails, there are a formula $\varphi(x;z)$ without parameters and a tuple $c$ in $\M$ such that
$X=\varphi(\M,c)\subset G(\M)$ is, in $\M$, a J-set but not weakly generic. As $M$ is $\aleph_0$-saturated, there is a tuple $c'$ in $M$ such that $\tp(c')=\tp(c)$.
Then there is an isomorphism $\sigma$ of $\M$ such that $c'=\sigma(c)$. Obviously, $X':=\varphi(\M,c')$ is, in $\M$, a J-set but not weakly generic. Then $X'(M)=\varphi(M,c')$ is, in $M$, a J-set but not weakly generic, which is a contradiction.
\end{proof}

Since whether J-sets coincide with weakly generic sets is invariant on $\aleph_0$-saturated models, we have

\begin{definition}
We say that a definable group $G$ has \textbf{J-property} if every weakly generic set of $G$ in $M$ is a J-set of $G$ in $M$ for some (any) $\aleph_0$-saturated model $M$.
\end{definition}

In particular, we can study the J-property in the monster $\M$. And for convenience, if we do not say in which model a definable set of $G$ is a J-set, then it is considered in $\M$.

\section{Positive results for commutative groups}\label{section-positive}

As $\AP_G(M)=\WGen_G(M)=\Gen_G(M)$ for stable groups, we study the relation between the properties on $S^\ext_G(M)$ and the place of the induced theory of $G$ in $T=\Th(M)$ in the classification theory. 

Now we show that every superstable commutative group has J-property. Note that as every type over $M$ is definable, every externally definable set is in fact definable and $S^\ext_G(M)=S_G(M)$. For $p\in S_G(M)$ and $N\prec M$, by $p|N$ we denote the restriction of $p$ on $N$.

\begin{theorem}\label{superstable}
Assume that $G$ is commutative and superstable. Then $G$ has J-property. 
\end{theorem}
\begin{proof}
Assume that $G$ is superstable and
$X=X(\M)$ is a definable J-set of $G$. For convenience, we assume that they are $0$-definable.

Let $p\in S_G(\M)$ be a generic global type with $p*p=p$. Let $(M_i)_{i\in\N}$ be an elementary increasing chain where $M_{i+1}$ is $|M_i|^+$-saturated for each $i$. 
Let $f_1\in G(M_1)^\N$ be a Morley sequence of $p|{M_0}$ and inductively, $f_{i+1}\in G(M_{i+1})^\N$ be a Morley sequence of $p|{M_{i}}$.

As $X$ is a J-set of $G$, for each $k\in\N$ and $\{f_i:1\leq i\leq k\}$, there are $m_k\in\N$, $a_k\in G$ and $t_k\in\jj_{m_k}$ such that for every $i\leq k$,
$$a_k+\sum_{j=1}^{m_k} f_i(t_k(j))\in X$$
Let $b_{i,k}=\sum_{j=1}^{m_k} f_i(t_k(j))$ for $i,k\in\N$ and $\varphi(x;y):=\text{``}x,y\in G \text{ and } x+y\in X\text{''}$. Then $\M\models\varphi(a_k, b_{i,k})$ when $i\leq k$.

\begin{claim-star}
For each $k\in\N$, the sequence $(b_{k,i})_{k\in\N}$ is a Morley sequence of $p$.
\end{claim-star}
\begin{claimproof}
Note that $b_{i,k}\in G(M_i)$ for every $i,k$.
As $p*p=p$, we have $p|{M_i}*p|{M_i}=p|{M_i}$ and consequently $\tp(b_{i+1,k}/M_i)=p|_{M_i}$ for every $i,k$. It is enough because $G$ is stable (see \cite{Stable} Fact 7.3).
\end{claimproof}

Then as $G$ is stable, by Proposition 7.6 in \cite{Stable}, there is $l\in\N$ which depends only on $\varphi(x;y)$ such that for every $k\in\N$,
either $|\{i\in\N: \M\models \varphi(a_k;b_{i,k})\}|<l$ or $|\{i\in\N: \M\models \neg\varphi(a_k;b_{i,k})\}|<l$.
We let $a=a_{l}$ and $b_i=b_{i,l}$ for $i\in\N$. Then $|\{i\in\N: \M\models \neg\varphi(a;b_{i})\}|<l$ and there is $l_1\in \N$ such that $\M\models\varphi(a;b_i)$ for every $i\geq l_1$. 

Then as $G$ is superstable, by Proposition 5.7 in \cite{Stable}, the extension of $\tp(a/M_{i+1})\supset \tp(a/M_i)$ forks only for finitely many $i$, and there is $l_2\in\N$ such that $\tp(a/M_{i+1})$ is nonforking over $M_i$ for $i\geq l_2$.

Take $l'> \max{l_1,l_2}$. Let $M=M_{l'}$ and $b=b_{l'+1}$. Then $a$ and $b$ are independent over $M$. Hence, 
$$\tp(a+b/M)=\tp(a/M)*\tp(b/M)=\tp(a/M)*p|M$$
which is generic. 
As $\M\models \varphi(a;b)$, we have $X(M)\in \tp(a+b/M)$. Hence, $X$ is generic. 
\end{proof}

As a corollary, every algebraic group in an algebraically closed field has J-property. So it is reasonable to study the J-property for definable groups in geometric fields. 
Now give two easy examples in $\pCF$, the theory of $\Q_p$ in the language of rings. More researches on this topic will be left to the future. As $\pCF$ is not stable, the examples shows that superstability is not necessary for J-property. 

When $T=\pCF$, we let $v$ be the valuation, $\cO=\cO_\M$ the valuation ring, and $\Gamma=\Gamma_\M$ the value group. It is well-known that $\pCF$ has cell decomposition and quantifier elimination in Macintyre's language which will not give more definable sets. More details can be found in \cite{pCF}.

\begin{prop}
Let $T=\pCF$ and $G=\Ga$, the additive group of $\M$. Then $G$ has J-property.
\end{prop}
\begin{proof}
Let $X$ be a definable J-set of $G$. Let $f_1, f_2\in G^\N$ such that $v(f_i(j))<v(f_i(j+1))$ and $v(f_1(j))<v(f_2(j))$ for every $i=1,2$ and $j\in\N$. Then there are $m\in\N$, $a\in G$ and $t\in\jj_{m}$ such that for $i=1,2$,
$a+\sum_{j=1}^{m} f_i(t(j))\in X$. Note that at least one of 
$a+\sum_{j=1}^{m} f_i(t(j))$ has valuation $\leq v(f_2(t(1)))$.
As for arbitrary $\alpha\in\Gamma_\M$, by saturation, we can let $v(f_2(j))\leq\alpha$ for every $j\in\N$, we have that $X$ is unbounded. 
It is known and easy to check that every unbounded definable set of $G$ is weakly generic which completes the proof.
\end{proof}

\begin{prop}
Let $T=\pCF$ and $G=(\cO,+)$, the additive group of the valuation ring of $\M$. Then $G$ has J-property.
\end{prop}
\begin{proof}
Let $X$ be a definable J-set of $G$. Let $k\in\N$ and $f_1, f_2\in G^\N$ such that $v(f_i(j))=i+j+k$ every $i=1,2$ and $j\in\N$. Then there are $m\in\N$, $a\in G$ and $t\in\jj_{m}$ such that for $i=1,2$,
$a+\sum_{j=1}^{m} f_i(t(j))\in X$. Note that 
$v((a+\sum_{j=1}^{m} f_1(t(j))-(a+\sum_{j=1}^{m} f_1(t(j)))=f_1(t(1))=1+k+t(1)$. As $k$ is arbitrary, we have $b_k,c_k\in X$ such that $v(b_k-c_k)\in \N$ and $v(b_k-c_k)\geq k$. 
By cell decomposition, we know that $X$ contains a translate of $p^l\cO$ for some $l\in\N$, which is known and easy to check to be weakly generic. 
\end{proof}

\section{An example without J-property}\label{section-example}

We follow \cite{example-neq}.
For $a\in\N$ with $a\geq 1$, it has a unique binary expansion $a=\sum_i \epsilon(i) 2^i$. We call $\{i\in\N: \epsilon(i)=1\}$ the \textbf{support} of $a$ and denote it by $\supp(a)$. Let $B_k=\{2^k,2^k+1,\dots, 2^{k+1}-1\}$ for $k\in \N$ and 
$$A=\{n\in \N: B_k\backslash \supp(n)\neq \emptyset \text{ for every } k\}$$
In \cite{example-neq}, $A\subset \N$ is shown to be a J-set (in the semigroup $\N$) but not piecewise syndetic. 
We do not give the original definition of piecewise syndetic here, but by Theorem 3.2 \cite{example-neq}, a subset $X$ of a (discrete) semigroup $S$ is \textbf{piecewise syndetic} iff there is $p$ in  the minimal two-sided ideal of $\beta S$ with $X\in p$. 

As we study groups, we transfer the example into a group. Obviously, $\beta\Z=\beta\N\cup (-\beta \N)$ where $-\beta \N:=\{-p:p\in \beta\N\}$ and $-p:=\{-X:X\in p\}$. We denote the minimal two-sided ideal of $\beta\N$ by $K$. Then clearly, $\hat K=K\cup(-K)$ is the minimal two-sided ideal of $\beta\Z$. Hence, $A$ is not piecewise syndetic in $(\Z,+)$. Note that $\beta\Z$ can be regard as the type space over $\Z$ in the theory that every subset of $\Z$ is definable, and in this theory, the notions of piecewise syndetic sets and weakly generic sets are the same. 

We let $\A$ be a unary predicate,
$T$ the theory of $M:=(\Z,+, A)$ with $A=\A(M)$, $\M=(\M,+,\A)$ a monster model and $G:=(\M,+)$. By restricting $p\in\beta\Z$ to externally definable sets in $M$, we have a surjective semigroup morphism $\pi:\beta\Z\rightarrow S^\ext_G(M)$. It is also a map of $\Z$-actions and by Lemma 1.4 \cite{Newelski}, $A$ is not weakly generic in $G(M)$. Then we have

\begin{prop}
$A$ is a J-set of $G$ in $M$ which is not weakly generic in $M$.
\end{prop}

For things on $\M$, we should study the proof of Lemma 5.2 \cite{example-neq} carefully. Note that from Remark 4.46 \cite{Hindman}, a subset $X$ of a semigroup $S$ is piecewise syndetic iff $\bigcup_{t\in H}\{s\in S: tx\in X\}$ is thick for some finite subset $H$ of $S$. Here, a subset $X$ of a semigroup $S$ is called \textbf{thick} if for every finite subset $H$ of $S$, there is $s\in S$ such that $Hs\subset X$ (Definition 4.45 \cite{Hindman}). Hence, to show that $A$ is not piecewise syndetic in $(\N,+)$, it is equivalent to showing that for every $c_1,\dots,c_m\in\N$, there are $b_1,\dots, b_n$ such that for every $x\in \N$, there is $y\in x+\{b_1,\dots, b_n\}$ with $(y+\{c_1,\dots,c_m\})\bigcap A=\emptyset$. 
We prove a better form that $n$ depends only on $m$.

\begin{lemma}
For every $m\in\N$. Let $n=2^{2^{m+1}}$. Then for every $c_1,\dots, c_m\in \N$ and $x\in \N$, there is $y\in x+\{1,\dots, n\}$ such that $(y+\{c_1,\dots,c_n\})\bigcap A=\emptyset$.
\end{lemma}
\begin{proof}
Let $c_1,\dots, c_m\in \N$ and $x\in\N$ be arbitrary. 

Inductively, we give $c'_k$ and $d_k$ for $1\leq k\leq m$.
Let $c'_1=c_1$ and
assume that $c'_1=\sum_j \epsilon_{1,j} 2^j$ where $\epsilon_{1,j}\in\{0,1\}$. Let $d_1=\sum_{j\in B_1}(1-\epsilon_{1,j}) 2^j$. Given $c'_k$ and $d_k$, we let $c'_{k+1}=c_k+d_k$. Assume that $c'_k=\sum_j \epsilon_{k,j} 2^j$ where $\epsilon_{k,j}\in\{0,1\}$. Let $d_{k+1}=(\sum_{1\leq j\leq k} d_k)+(\sum_{j\in B_{k+1}}(1-\epsilon_{k,j})2^j)$. 

Let $d=d_m$.
It is easy to see that $\supp(d+c_k)\supset B_k$ for every $1\leq k\leq m$. Clearly, there always exists $y\in x+\{1,\dots,n\}$ such that $y\equiv d$ (mod $n=2^{2^{m+1}}$). Then $y+ c_k\equiv d+c_k$ (mod $n$) and $\supp(y+c_k)\bigcap\{1,\dots,2^{m+1}-1\}=\supp(d+c_k)\bigcap\{1,\dots,2^{m+1}-1\}\supset B_k$ for each $k$. Then $(y+\{c_1,\dots,c_n\})\bigcap A=\emptyset$. 
\end{proof}

The advantage of the above from is that for given $m$ and $n=2^{2^{m+1}}$, the statement is first order and can be transferred. Then we have

\begin{prop}
$\A$ is not weakly generic for $G$ in $\M$.
\end{prop}
\begin{proof}
To make use of the structure $(\N,+)$, we add $<$ into the language of $T$. Then we may assume that $(\M^{>0},+,\A)$ is a monster model of $(\N,+,\A)$. Then by transferring the above lemma for each $m$ to $(\M^{>0},+)$, we have that $\A$ is not piecewise syndetic in $(\M^{>0},+)$. Like the above argument from $(\N,+)$ to $(\Z,+)$, we have that $\A$ is not piecewise syndetic in $(\M,+)$. Also by the same argument above, we have that $\A$ is not weakly generic for $G$ in $\M$.
\end{proof}

Now we show that $\A$ is a J-set of $G$ in $\M$.
Like Lemma 5.1 \cite{Hindman}, we have 
\begin{lemma}\label{Hind-1}
Let $f_1,\dots,f_n\in \M^\N$. Then there is $m\in\N$ and $t\in\jj_m$ such that $\sum_{j\in t}f_i(j)\in 2^{n+1}\M$ for each $i$.
\end{lemma}
\begin{proof}
Note that $\Z/2^{n+1}\Z$ is finite, and we have $\M/2^{n+1}\M=\Z/2^{n+1}\Z$. Let $I_0=\N$. Inductively, we can find $I_{i}\subset I_{i-1}$ for $i=1,\dots,n$ such that for every $j_1,j_2\in I_i$, $f_i(j_1)-f_i(j_2)\in 2^{n+1}\M$. Let $m=2^{n+1}$ and $t\in \jj_m$ with $t\subset I_n$. It is easy to check that these $m$ and $t$ are what we want. 
\end{proof}

Then we modify the proof of Lemma 5.2 \cite{example-neq} which is original on $(\N,+)$ to $(\Z,+)$.
\begin{lemma}\label{Hind-2}
Let $b_1,\dots,b_n\in 2^{n+1}\Z$. Then there is $a\in\Z$ such that $a+b_i\in A$ for every $i$.
\end{lemma}
\begin{proof}
Let $b\in 2^{n+1}\N$ such that $b>|b_i|$ for each $i$. Let $c_{0,i}=b+b_i$ for each $i$. Note that $2^{n+1}\leq c_{0,i}\in 2^{n+1}\N$ for each $i$. The next paragraph is from the proof of Lemma 5.2 \cite{example-neq}.

Let $l\in\N$ be the least such that $2^l>n$. As $2^{l-1}<n+1$, we have $2^{l-1}\in B_{l-1}\backslash \supp(c_{0,i})$. As $|B_l|=2^l>n$, there is $r_0\in B_l$ such that
$B_l\backslash \supp(2^{r_0}+c_{0,i})\neq\emptyset$ for each $i$. Let $c_{1,i}=2^{r_0}+c_{0,i}$. Inductively, given $c_{j,i}$, as $|B_{l+j}|=2^{l+j}>n$, there is $r_j\in B_{l+j}$ such that
$B_{l+j}\backslash \supp(2^{r_j}+c_{j,i})\neq\emptyset$ for each $i$. We can stop at $j=k$ when $2^{2^{k+i}}>\max_i c_{0,i}$. Let $c=\sum_{j=0}^{k}2^{r_j}$. It is easy to check that $c+c_{0,i}\in A$ for each $i$.

Let $a=b+c$. Then $a+b_i\in A$ for each $i$.
\end{proof}

It is clear that the statement of Lemma \ref{Hind-2} is first order on each $n$ and hence holds in $(\M,+,\A)$. Combining with Lemma \ref{Hind-1}, we have

\begin{prop}
$\A$ is a J-set of $G$ in $\M$.
\end{prop}

Hence, witnessed by $\A$, we have
\begin{prop}
The definable group $G$ in $T$ does not have J-property. 
\end{prop}

However, this theory $T$ is not stable:

\begin{prop}
The upper Banach density $\delta$ of $A$ in $\N$ is $>0$. Hence, by Theorem D \cite{GC}, the theory $T$ is not stable.
\end{prop}
\begin{proof}
It is easy to see that for $d\geq 1$, 
$$\frac{|A\cap\{0,1,\dots, 2^{2^{d+1}}-1\}|}{2^{2^{d+1}}}=\frac{2\prod_{k=0}^d(2^{2^k}-1)}{2^{2^{d+1}}}=\prod_{k=0}^d(1-2^{-2^k})$$
Clearly, the limit $\prod_{k=0}^\infty(1-2^{-2^k})$ exists and $\delta\geq \prod_{k=0}^\infty(1-2^{-2^k})$.
To show $\prod_{k=0}^\infty(1-2^{-2^k})>0$, it suffices to show $\sum_{k=0}^\infty \log(1-2^{-2^k})$  converges. Note that $\log(1+x)=x-\frac{1}{2}x^2+\frac{1}{3}x^3-\frac{1}{4} x^4+\dots$ when $|x|<1$. Then $|\log(1+x)|\leq |x|+|x|^2+|x|^3+\dots=\frac{|x|}{1-|x|}$ for $|x|<1$. 
Then for each $k$,
$$|\log(1-2^{-2^k})|\leq \frac{2^{-2^k}}{1-2^{-2^k}}=\frac{1}{2^{2^k}-1}\leq 2^{-2^k-1}$$
Hence, $\sum_{k=0}^\infty \log(1-2^{-2^k})$ converges absolutely. 
\end{proof}

As we have shown that every superstable commutative group has J-property and it is well-known that $(\Z,+)$ is superstable, it is natural to ask:

\begin{question}
Is there a stable commutative group without J-property? Or more specifically, is there a stable expansion of the superstable group $(\Z,+)$ by a unary predicate for $A\subset \N$ such that the expansion does not have J-property witnessed by $\A:=A(\M)$. 
\end{question}

\begin{remark}
Theorem 7.16 \cite{graph} gives a method to find stable expansions. It says that $(\Z,+,A)$ is stable iff $(V,E)$ is stable where $V\subset\N$ is vaporous, $E$ is a graph relation on $V$ and $A=V\cup \{x+y:x Ey\}$.
 Note that when $E=\emptyset$ or $E=V\times V$, $(V,E)$ is superstable. Let $A_1=V$ and $A_2=V\cup (V+V)$. Let $\M$ be a monster model of $(\Z,+, V, A ,A_1,A_2)$. We can study $(\Z,+,A)$ by restricting the language. We use blackboard bold style to denote the global objects. 

\noindent $\bullet$ If $\V$ is a J-set, by applying Theorem \ref{superstable} to $(\Z,+,A_1)$, $\A_1=\V$ is also weakly generic. In this case, $\A$ is weakly generic.

\noindent $\bullet$ Assume that $\V$ is not a J-set. If $\A_2$ is not a J-set, then $\A$ is not a J-set and not weakly generic. So we assume that $\A_2$ is a J-set. By  \cite{Hindman} Lemma 14.14.6, $\V+\V$ is a J-set. But it is obvious from the definition of J-sets that $\V$ is a J-set, which is a construction. 

Hence, this construction by $(V,E)$ does not solve the question. 
\end{remark}


\begin{thebibliography}{99}
\bibitem{pCF}
L. B\'{e}lair, Panorama of $p$-adic model theory, Annales des sciences math\'{e}matiques du Qu\'{e}bec, 2011.
\bibitem{Hindman} N. Hindman, D. Strauss, Algebra in the Stone-Cech Compactification: Theory and Applications.
\bibitem{example-neq} N. Hindman and A. Maleki, Central Sets and Their Combinatorial Characterization, Journal of combinatorial theory, Series A 74, 188-208 (1996).
\bibitem{Newelski} L. Newelski, Topological Dynamics of Definable Group Actions, The Journal of Symbolic Logic, 2009, Vol. 74, No. 1, pp. 50-72.
\bibitem{Stable} A. Pillay, An Introduction to Stability Theory.
\bibitem{GC} G. Conant, Stability and sparsity in sets of natural numbers. Israel J. Math., 230(1):471–508, 2019.
\bibitem{graph} G. Conant, C. d'Elb\'{e}e, Y. Halevi, L. Jimenez, S. Rideau-Kikuchi, Enriching a predicate and tame expansions of the integers, arXiv:2203.07226v3 [math.LO].

\end{thebibliography}
\end{document}